\documentclass[12pt,leqno]{amsart}
\usepackage{amsmath}
\usepackage{amssymb}

\def\underset#1#2{{\mathrel{\mathop {{}_{} {#2}}\limits_{{#1}_{}}}}}
\def\upplim_#1{\underset{#1}{\overline\lim}\;}
\def\lowlim_#1{\underset{#1}{\underline\lim}\;}
\newtheorem{lemma}[equation]{Lemma}

\newtheorem{remark}[equation]{Remark}
\newtheorem{theorem}[equation]{Theorem}

\newcommand{\C}{{\mathbb{C}}}

\newcommand{\supp}{\mathrm{Supp}\,}

\newcommand{\Z}{\mathbb{Z}}

\numberwithin{equation}{section}

\title[Second main theorem and uniqueness problem of meromorphic functions]{Second main theorem and uniqueness problem of meromorphic functions with finite growth index sharing five small functions on a complex disc} 

\date{ }
\author{Si Duc Quang}

\begin{document}

\begin{abstract}
This paper has twofold. The first is to establish a second main theorem for meromorphic functions on the complex disc $\Delta (R_0)\subset\C$ with finite growth index and small functions, where the counting functions are truncated to level $1$ and the small term is more detailed estimated. The second is to prove a generalization and improvement of the five values theorem of Nevanlinna for the case of five small functions on the complex disc $\Delta(R_0)$.
\end{abstract}

\def\thefootnote{\empty}
\footnotetext{
\hskip8pt This research is funded by Vietnam National Foundation for Science and Technology Development (NAFOSTED) under grant number 101.02-2021.12.}

\maketitle

\section{Introduction and Results}

In $1926$, Nevanlinna \cite{N} showed that two distinct non-constant meromorphic functions $f$ and $g$ on $\C$ cannot have the same inverse images for five distinct values. This theorem is called the five values theorem. Later on, many authors have improved and generalized this theorem by replacing the five values by five small functions and by weakening the condition ``having the same inverse image''. For related results, we refer the readers to some important papers such as  \cite{Q}, \cite{Yao}, \cite{Yi}, \cite{YJ}, and the references therein. All of proofs of these results are based on the second main theorem of Yi-Yang \cite{YY} or the second main theorem of Yamanoi \cite{Y} for meromorphic functions and small functions on $\C$.

Recently, Ru and Sibony \cite{RS} gave a definition of a class of meromorphic functions on the complex disc $\Delta(R_0)=\{z\in\C;|z|<R_0\}$ with finite growth index, which is larger than the class of transcendental meromorphic functions. They also established a new second main theorem for meromorphic functions in this class with fixed values. Motivated by their work, in this paper we will give a truncated second main theorem for meromorphic functions of finite growth index on $\Delta(R_0)$ with small functions, which generalizes the second main theorem of Yi-Yang mentioned above.

In order to state our result, we recall the following notation. Let $\Delta (R_0)=\{z\in\C ; |R|<R_0\}\ (0<R_0\le +\infty)$ be a disc in $\C$. Let $\nu$ be a divisor on $\Delta (R_0)$, which is regarded as a function on $\Delta (R_0)$ with values in $\Z$ such that $\supp (\nu):=\{z;\nu(z)\ne 0\}$ is a discrete subset of $\Delta (R_0)$. The counting function of $\nu$ is defined by:
$$ n(t):=\sum_{|z|\le t}\nu(z) \ (0\le t\le R_0) \text{ and }\ N(r,\nu):=\int_{0}^{r}\dfrac{n(t)-n(0)}{t}dt.$$

Let $f:\Delta (R_0)\rightarrow \C\cup\{\infty\}$ be a non-constant meromorphic function. We denote by $\nu^0_f$ (resp. $\nu^\infty_f$) the divisor of zeros (resp. divisor of poles) of $f$. For a meromorphic function $a$, we define
$$ N(r,a,f):=N(r,\nu^0_{f-a})\text{ and }\overline{N}(r,a,f):=N(r,\min\{1,\nu^0_{f-a}\}).$$
Here, if $a\equiv\infty$ then we regard $\nu^0_{f-\infty}$ as $\nu^\infty_f$. We say that the function $N(r,a,f)$ (resp. $\overline{N}(r,a,f)$) is the counting function with multiplicity (resp. the counting function without multiplicity) of the divisor of all $a$-points of $f$. 

The proximity function and the characteristic function of $f$ is defined respectively by
\begin{align*}
m(r,f)&:=\int_{0}^{2\pi}\log^+|f(re^{i\theta})|d\theta.\\
\text{and }T(r,f)&:=m(r,f)+N(r,\infty,f).
\end{align*}
A meromorphic function $a$ is said to be small with respect to $f$ if $T(r,a)=o(T(r,f))$ as $r\rightarrow R_0$. The first main theorem states that
$$ T(r,f)=T\left(r,\frac{1}{f-a}\right)+o(T(r,f))$$
for every small function $a$ (here, we regard $\frac{1}{f-\infty}$ as $f$).

According to Ru and Sibony \cite{RS}, the growth index of $f$ is defined by
$$ c_{f}=\mathrm{inf}\left\{c>0;\int_{0}^R\mathrm{exp}(cT(r,f))dr=+\infty\right\}.$$
For convenient, we will set $c_f=+\infty$ if $\left\{c>0;\int_{0}^R\mathrm{exp}(cT(r,f))dr=+\infty\right\}=\varnothing$. 

For the case where the domain $\Delta(R_0)$ is $\C$, Yi and Yang proved a second main theorem for small functions as follows.

\vskip0.2cm
\noindent
\textbf{Theorem A} (see \cite[p. 185]{YY}) {\it Let $f$ be a non-constant meromorphic function on $\C$ and let $a_1,\ldots,a_5$ be five distinct small functions (with respect to $f$). Then we have
$$ \|\ 2T(r,f)\le\sum_{i=1}^5\overline{N}(r,a_i,f)+o(T(r,f)).$$}

Here, the notation ``$\|\ P$'' means the assertion $P$ holds for all $r\in (0,+\infty)$ outside a finite Borel measure set. Hence, for $q$ distinct small functions $a_1,\ldots,a_q\ (q\ge 5)$, from the above second main theorem we get
$$ \|\ \dfrac{2q}{5}T(r,f)\le\sum_{i=1}^q\overline{N}(r,a_i,f)+o(T(r,f)).$$ 
We note that, Yamanoi in \cite{Y} improved the above result by increasing the coefficient $\frac{2q}{5}$to $(q-2)$. The result of Yamanoi is believed to be sharp for the case of functions on $\C$. However, the method of Yamanoi in \cite{Y} does not work for the case of functions on complex discs with finite growth index.

Our first purpose in this paper is to generalize Theorem A to the case of meromorphic functions on $\Delta(R_0)$ with finite growth index. Namely, we will prove the following.
\begin{theorem}\label{1.1}
Let $f$ be a non-constant meromorphic function on $\Delta (R_0)$ and let $a_1,\ldots,a_5$ be five distinct small functions (with respect to $f$). Let $\gamma(r)$ be a non-negative measurable function defined on $(0,R_0)$ with $\int_0^{R_0}\gamma (r)dr=+\infty$. Then, for every $\varepsilon >0$,
$$ \|_E\ 2T(r,f)\le\sum_{i=1}^5\overline{N}(r,a_i,f)+19((1+\varepsilon)\log\gamma(r)+\varepsilon\log r)+o(T(r,f)).$$
\end{theorem}
Here and later on, we use the notation $\|_E P$ to say that the assertion $P$ holds for all $r\in (0;R_0)$ outside a subset $E$ of $(0;R_0)$ with $\int_E\gamma(r) dr <+\infty$.
\begin{remark}{\rm

(1) If the characteristic function $T(r,f)$ grow rapidly enough (more than the oder of $\log\frac{1}{R_0-r}$) then Theorem \ref{1.1} is not more special than Theorem A. But for the general case, Theorem \ref{1.1} will be more interesting since the error term in the second main theorem may not be small term $o(T(r,f))$.

(2) If we assume that $f$ is of finite growth index (i.e., $c_f<+\infty$) then in Theorem \ref{1.1} we may take $\gamma (r)=\exp{(c_f+\varepsilon)T(r,f)}$.

(3) For the case of $R_0=+\infty$, we may take $c_f=0$ (with non-constant meromorphic function $f$) and Theorem \ref{1.1} will give back Theorem A.}
\end{remark}

Our second aim in this paper is applying the above second main theorem to investigate the uniqueness problem of meromorphic functions on $\Delta(R_0)$ with finite growth index sharing some small functions ignoring multiplicity. 

Now, for two meromorphic functions $h$, $a$ and a positive integer $k$ ($k$ may be $+\infty$), we denote by $\overline{E}_{k)}(a,h)$ the set of all zeros of the function $h-a$ with multiplicity not exceed $k$ (if $a\equiv\infty$ then $\overline{E}_{k)}(a,h)$ is regarded as the set of all poles of $h$ with with multiplicity not exceed $k$). Our uniqueness is stated as follows.

\begin{theorem}\label{1.2}
Let $f$ and $g$ be a non-constant meromorphic functions on $\Delta (R_0)$ with finite growth indices $c_f,c_g$. Let $a_1,\ldots,a_q$ be $q\ (q\ge 5)$ distinct small functions (with respect to $f$ and $g$) and let $k$ be a positive integer or $k=+\infty$. Assume that 
$$\overline{E}_{k)}(a_i,f)=\overline{E}_{k)}(a_i,g)\ (1\le i\le q).$$
If $c_f+c_g<\dfrac{k(2q-8)-3(q+4)}{k(19q-\frac{117}{2})+19(q+4)}$ then $f=g$.
\end{theorem}
\begin{remark}{\rm

(1) If $c_f=c_g=0$ then the condition of the above theorem is fulfilled with 
$$k>\dfrac{3(q+4)}{2q-8}.$$ 
Moreover, for the case of $q=5$, we only need $k\ge 14$. Therefore, our result improves  the result of Yao in \cite{Yao} (for $q=5,k\ge 22$) and generalizes the result of Yi in \cite{Yi} (for $q=5$ and $k\ge 14$).

(2) If $k=+\infty$ and $q=5$ then the condition of the above theorem is fulfilled with 
$$c_f+c_g<\dfrac{4}{73}.$$ 
Then, we have a generalization of the classical five values theorem of Nevanlinna.

(3) In the proof of Theorem \ref{1.2}, we use the method of Yi in \cite{Yi}. However, the most difficult comes from the fact that all error terms occurring are not small with respect to the functions. Therefore, we will have to control such terms at every step in the proof.}
\end{remark}

\section{Truncated second main theorem for meromorphic functions and small functions on the complex disc}
In this section, we will proof Theorem \ref{1.1}. Firstly, we need the following lemma on logarithmic derivative due to Ru and Sibony \cite{RS}.

\begin{lemma}[{Lemma on logarithmic derivative \cite[Theorem 5.1]{RS}}]\label{2.1}
Let $0<R\le +\infty$ and let $\gamma (r)$ be a non-negative measurable function defined on $(0,R_0)$ with $\int_0^{R_0}\gamma (r)dr=\infty$. Let $f(z)$ be a meromorphic function on $\Delta (R_0)$. Then, for $\varepsilon >0$, we have
$$\bigl\|_E\  m\left (r,\frac{f'}{f}\right)\le (1+\varepsilon)\log\gamma (r)+\varepsilon\log r+O(\log T(r,f)).$$
\end{lemma}
Then, for any small function $a$ (with respect to $f$) we also have 
\begin{align}\label{2.2}
\bigl\|_E\ m\left (r,\frac{a'}{a}\right)\le (1+\varepsilon)\log\gamma (r)+\varepsilon\log r+o(T(r,f)).
\end{align}
This implies that
\begin{align}\label{2.3}
\begin{split}
\biggl\|_E\ N\left (r,0,\frac{a'}{a}\right)&\le T\left (r,\frac{a'}{a}\right)=N\left (r,\infty,\frac{a'}{a}\right)+m\left (r,\frac{a'}{a}\right)\\
&\le \overline{N}(r,0,a)+\overline{N}(r,\infty,a)+(1+\varepsilon)\log\gamma (r)+\varepsilon\log r+o(T(r,f))\\
&=(1+\varepsilon)\log\gamma (r)+\varepsilon\log r+o(T(r,f)).
\end{split}
\end{align}
 Ru and Sibony proved the following second main theorem for fixed values.
\begin{theorem}[{see \cite[Theorem 1.7]{RS}}] \label{2.4}
Let $f $ be a non-constant meromorphic function on $\Delta (R_0)\ (0<R_0\le +\infty)$. Let $\gamma (r)$ be a non-negative measurable function defined on $(0,R_0)$ with $\int_0^{R_0}\gamma (r)dr=\infty$ and let $c_1,\ldots,c_q$ be $q$ distinct values in $\C\cup\{\infty\}$. Then, for every $\varepsilon >0$, we have
\begin{align*}
\bigl\|_E\ (q-n-1)T(r,f)\le&\sum_{j=1}^q\overline{N}(r,c_i,f)+(1+\varepsilon)\log\gamma (r)+\varepsilon\log r+o(T(r,f)).
\end{align*}
\end{theorem}

\begin{proof}[\sc Proof of Theorem \ref{1.1}]
We set $S(r)=(1+\varepsilon)\log\gamma (r)+\varepsilon\log r$. Replacing $f$ by its quasi-M\"{o}bius transformation
$$g = \frac{f-a_2}{f-a_1} \cdot \frac{a_3-a_1}{a_3-a_2}$$
if necessary, we may assume that $a_1=\infty,a_2=0,a_3=1$ and $a_4=b_1,a_5=b_2$, where $b_1,b_2$ are two small functions with respect to $f$.

If $b_1$ or $b_{2}$ is constant then we get the conclusion from the second main theorem for four values (Theorem \ref{2.4}). Then, we may assume that both $b_1$ and $b_2$ are not constant. We define
\begin{equation}\label{2.5}
F= \begin{vmatrix}
ff'&f'&f^2-f \\
b_1b_1'&b_1'&b_1^2-b_1 \\
b_2b_2' &b_2'&b_2^2-b_2
\end{vmatrix}.
\end{equation}
We consider the following two cases.

\textbf{Case 1: }$F(z)\equiv0$. From (\ref{2.5}) we get
\begin{align}\label{2.6}
\biggl (\dfrac{b_1'}{b_1}-\dfrac{b_2'}{b_2}\biggl )\biggl (\dfrac{f'}{f-1}-\dfrac{b_2'}{b_2-1}\biggl )\equiv \biggl (\dfrac{b_1'}{b_1-1}-\dfrac{b_2'}{b_2-1}\biggl )\biggl (\dfrac{f'}{f}-\dfrac{b_2'}{b_2}\biggl ).
\end{align}

We distinguish the following four subcases. 

\emph{Subcase 1:} $\dfrac{b_1'}{b_1}\equiv\dfrac{b_2'}{b_2}.$ We have $\dfrac{b_1'}{b_1-1}\not\equiv\dfrac{b_2'}{b_2-1}$, since if otherwise $b_1$ and $b_{2}$ are constants. Therefore, from (\ref{2.6}), we have $\dfrac{f'}{f}\equiv\dfrac{b_2'}{b_2}$, and hence $f=cb_2$ with a constant $c$. This implies that $T(r,f)=o(T(r,f))$ and we get a contradiction. 

\emph{Subcase 2:} $\dfrac{b_1'}{b_1-1}\equiv\dfrac{b_2'}{b_2-1}.$ By the same arguments in Subcase 1, we get again the contradiction. 

\emph{Subcase 3:} $\dfrac{b_1'}{b_1}\not\equiv\dfrac{b_2'}{b_2}, \dfrac{b_1'}{b_1-1}\not\equiv\dfrac{b_2'}{b_2-1},\dfrac{b_1'}{b_1}-\dfrac{b_2'}{b_2} \equiv\dfrac{b_1'}{b_1-1}-\dfrac{b_2'}{b_2-1}.$ The identity (\ref{2.6}) implies that  
$$\dfrac{f'}{f-1}-\dfrac{f'}{f}\equiv \dfrac{b_2'}{b_2-1}-\dfrac{b_2'}{b_2}.$$
Then, we have
$$\dfrac{f-1}{f}\equiv c \cdot\dfrac{b_2-1}{b_2},$$
where $c$ is a constant. Therefore $f=\frac{b_2}{(c-1)b_2-c}$, and hence $T(r,f)=o(T(r,f))$. This is a contradiction.

\emph{Subcase 4:} $\dfrac{b_1'}{b_1}\not\equiv\dfrac{b_2'}{b_2}, \dfrac{b_1'}{b_1-1}\not\equiv\dfrac{b_2'}{b_2-1},\dfrac{b_1'}{b_1}-\dfrac{b_2'}{b_2} \not\equiv\dfrac{b_1'}{b_1-1}-\dfrac{b_2'}{b_2-1}.$ The identity (\ref{2.6}) may be rewritten as
\begin{align}\label{2.7}
\biggl (\dfrac{b_1'}{b_1}-\dfrac{b_2'}{b_2}\biggl )\dfrac{f'}{f-1}-\biggl (\dfrac{b_1'}{b_1-1}-\dfrac{b_2'}{b_2-1}\biggl )\dfrac{f'}{f}\equiv \dfrac{b_1'b_2'}{b_1(b_2-1)}-\dfrac{b_2'b_1'}{b_2(b_1-1)}.
\end{align}
From (\ref{2.7}), we see that each zero of $(f-1)$ must be a zero or an $1$-point or a pole of $b_j\  (j=1, 2)$ or a zero of $\dfrac{b_1'}{b_1}-\dfrac{b_2'}{b_2}$. Therefore,
\begin{align}\label{2.8}
\min\{1,\nu^0_{f-1}\}\le\sum_{i=1,2}\sum_{a=0,1,\infty}\min\{1,\nu^{0}_{b_i-a}\}+\min\{1,\nu^0_{\frac{b_1'}{b_1}-\frac{b_2'}{b_2}}\} .
\end{align}
Similarly, we have
\begin{align}\label{2.9}
\min\{1,\nu^0_{f}\}\le\sum_{i=1,2}\sum_{a=0,1,\infty}\min\{1,\nu^{0}_{b_i-a}\}+\min\{1,\nu^0_{\frac{b_1'}{b_1-1}-\frac{b_2'}{b_2-1}}\} .
\end{align}
From (2.7), we also see that each pole of $f$ must be a zero or an $1-$point or a pole of $b_j\  (j=1, 2)$ or a zero of $\biggl (\dfrac{b_1'}{b_1}-\dfrac{b_2'}{b_2}\biggl )- \biggl (\dfrac{b_1'}{b_1-1}-\dfrac{b_2'}{b_2-1}\biggl )$. Therefore, by the similar arguments, we have
\begin{align}\label{2.10}
\min\{1,\nu^\infty_{f}\}\le\sum_{i=1,2}\sum_{a=0,1,\infty}\min\{1,\nu^{0}_{b_i-a}\}+\min\{1,\nu^0_{\bigl (\frac{b_1'}{b_1}-\frac{b_2'}{b_2}\bigl )- \bigl (\frac{b_1'}{b_1-1}-\frac{b_2'}{b_2-1}\bigl )}\}.
\end{align}
Combining (\ref{2.8}), (\ref{2.9}) and (\ref{2.10}), we have 
\begin{align}\label{2.11}
\begin{split}
\sum_{a=0,1,\infty}\min\{1,\nu^0_{f-a}\}&\le\sum_{i=1,2}\sum_{a=0,1,\infty}\min\{1,\nu^{0}_{b_i-a}\}+\min\{1,\nu^0_{\frac{b_1'}{b_1}-\frac{b_2'}{b_2}}\}\\
&+\min\{1,\nu^0_{\frac{b_1'}{b_1-1}-\frac{b_2'}{b_2-1}}\}+\min\{1,\nu^0_{\bigl (\frac{b_1'}{b_1}-\frac{b_2'}{b_2}\bigl )- \bigl (\frac{b_1'}{b_1-1}-\frac{b_2'}{b_2-1}\bigl )}\}.
\end{split}
\end{align}
Then by using the second main theorem (Theorem \ref{2.4}), we get
\begin{align*}
\|_E\ T( r,f)&\le \overline{N}(r,0,f)+\overline{N}(r,1,f)+\overline{N}(r,\infty,f)+S(r)\\
&\le\sum_{i=1,2}\sum_{a=0,1,\infty}\overline{N}(r,0,b_i-a)+\overline{N}\left(r,0,\frac{b_1'}{b_1}-\frac{b_2'}{b_2}\right)+\overline{N}\left(r,0,\frac{b_1'}{b_1-1}-\frac{b_2'}{b_2-1}\right)\\
&+\overline{N}\left(r,0,\left(\frac{b_1'}{b_1}-\frac{b_2'}{b_2}\right )- \left (\frac{b_1'}{b_1-1}-\frac{b_2'}{b_2-1}\right)\right)+S(r)\\
&\le 4S(r)  \ \ \ \text{ (by (\ref{2.3}))}.
\end{align*}
Then we have the desired second main theorem in this case.

\noindent
\textbf{Case 2:} $F(z)\not\equiv0$. We set
\begin{align*}
\delta (z)&= \min\{1,|b_1(z)|,|b_2(z)|,|b_1(z)-1|,|b_2(z)-1|,|b_1(z)-b_2(z)|\},\\
\theta _j (r)& = \{\theta:|f(re^{i\theta})-b_j(re^{i\theta})|\le\delta (re^{i\theta})\},(j=1,2),\\
\theta _3(r)& = \{\theta:|f(re^{i\theta})|\le\delta (re^{i\theta})\},\\
\theta _4(r)&= \{\theta:|f(re^{i\theta})-1|\le\delta (re^{i\theta})\}.
\end{align*}
It is clear that the sets $\theta_i(r)\cap\theta_j(r) \ ( i\neq j, i, j=1, 2, 3, 4)$ have at most many finite points. Then, we easily see that  
\begin{align*}
\dfrac{1}{2\pi}\int\limits_{0}^{2\pi}\log \dfrac{1}{\delta (re^{i\theta})}d\theta
&\le\dfrac{1}{2\pi}\int\limits_{0}^{2\pi}\log \max  \left\{\ 1,\dfrac{1}{|b_1|},\dfrac{1}{|b_2|},\dfrac{1}{|b_1-1|},\dfrac{1}{|b_2-1|},\dfrac{1}{|b_1-b_2|}\right\}d\theta\\
&\le m\left( r,\dfrac{1}{b_1}\right)+m\left( r,\dfrac{1}{b_2}\right)+m\left( r,\dfrac{1}{b_1-1}\right)\\
&+m\left( r,\dfrac{1}{b_2-1}\right)+m\left( r,\dfrac{1}{b_1-b_2}\right).
\end{align*}
Therefore,
\begin{align}\label{2.12}
\dfrac{1}{2\pi}\int\limits_{0}^{2\pi}\log \dfrac{1}{\delta (re^{i\theta})}d\theta=o(T(r,f)).
\end{align}
We also see that
\begin{align*}
ff'&=(f-b_1)(f'-b_1')+b_1'(f-b_1)+b_1(f'-b_1')+b_1b_1',\\
f'&=(f'-b_1')+b_1',\\
f^2-f&=(f-b_1)^2+(2b_1-1)(f-b_1)+b_1^2-b_1.
\end{align*}
Substituting these functions (on the right hand side) into (\ref{2.5}), by the properties of the Wronskian, we have
\begin{equation}\label{2.13}
F= \begin{vmatrix}
\varphi&f'-b_1'&\psi \\
b_1b_1'&b_1'&b_1'-b_1 \\
b_2b_2' &b_2'&b_2'-b_2
\end{vmatrix}, 
\end{equation}
where
\begin{align*}
\varphi &=(f-b_1)(f'-b_1')+b_1'(f-b_1)+b_1(f'-b_1'),\\
\psi&=(f-b_1)^2+(2b_1-1)(f-b_1).
\end{align*}
From (\ref{2.13}) we see that each zero with multiplicity $p$ $(p>1)$ of $f-b_1$ which is neither a pole of $b_1$ nor a pole of $b_2$ must be a zero of  $F$ with multiplicity at least $p-1$. Similarly, each zero with multiplicity $p$ $(p>1)$ of $f-b_2$ which is neither a pole of $b_1$ nor a pole of $b_2$ must be a zero of $F$ with multiplicity at least $p-1$. Moreover, from (\ref{2.5}) one see that each zero of multiplicity $p$ $(p>1)$ of $f$ or $f-1$ which is neither a pole of  $b_1$ nor a pole of $b_2$ must be a zero of $F$ with multiplicity at least $p-1$. This implies that
\begin{align}\label{2.14}
\sum_{i=2}^5(N(r,a_i,f)-\overline{N}(r,a_i,f))\le N(r,0,F).
\end{align}

On the other hand, since $|f(re^{i\theta})-b_1(re^{i\theta})|\le\delta(re^{i\theta})\le 1+|b_1(re^{i\theta})|$ for every $\theta\in\theta_1(r)$, 
\begin{align}\label{2.15}
\begin{split}
\bigl\|_E\ \frac{1}{2\pi }&\int\limits_{\theta_1(r)}\log^+\left| \dfrac{F}{f-b_1}\right|d\theta\le \frac{1}{2\pi }\int\limits_{\theta_1(r)}\log\left (1+\left| \dfrac{F}{f-b_1}\right|\right)d\theta\\
&\le \frac{1}{2\pi }\int\limits_{\theta_1(r)}\log\left\{(1+|b_1'|)(1+|b_1|)\left(1+\left|\frac{f'-b_1'}{f-b_1}\right|\right)(1+|f-b_1|)^2\right\}d\theta\\
&+\frac{1}{2\pi }\int\limits_{\theta_1(r)}\log\left\{(1+|b_1|)(1+|b_1'|)\right\}d\theta\\
&+\frac{1}{2\pi }\int\limits_{\theta_1(r)}\log\left\{(1+|b_2|)(1+|b_2'|)\right\}d\theta+O(1)\\
&\le m\left( r,\dfrac{f'-b_1'}{f-b_1}\right)+2m\left(r,\frac{b_1'}{b_1}\right)+m\left(r,\frac{b_2'}{b_2}\right)+o(T(r,f))\\
&\le 4S(r)+o(T(r,f)).
\end{split}
\end{align}
Therefore, from (\ref{2.11}) and (\ref{2.15}) we have
\begin{align}\label{2.16}
\begin{split}
\bigl\|_E\ m\left( r,\dfrac{1}{f-b_1}\right)&\le \frac{1}{2\pi } \int\limits_{\theta_1(r)}\log^+\left| \dfrac{1}{f-b_1} \right|d\theta +\dfrac{1}{2\pi}\int\limits_{0}^{2\pi}\log \dfrac{1}{\delta(re^{i\theta})}d\theta\\
&\le \frac{1}{2\pi } \int\limits_{\theta_1(r)}\log^+\left| \dfrac{F}{f-b_1} \right|d\theta +\frac{1}{2\pi } \int\limits_{\theta_1(r)}\log^+\dfrac{1}{|F|}d\theta+o(T(r,f))\\
&= \frac{1}{2\pi } \int\limits_{\theta_1(r)}\log^+ \dfrac{1}{|F|}d\theta+4S(r)+o(T(r,f)). 
\end{split}
\end{align}
Similarly, we get
\begin{align}
\label{2.17}
\bigl\|_E\ m\left( r,\dfrac{1}{f-b_2}\right)&\le \frac{1}{2\pi } \int\limits_{\theta_2(r)}\log^+ \dfrac{1}{|F|}d\theta+4S(r)+o(T(r,f)),\\
\label{2.18}
\bigl\|_E\ m\left( r,\dfrac{1}{f}\right)&\le \frac{1}{2\pi } \int\limits_{\theta_3(r)}\log^+ \dfrac{1}{|F|}d\theta+4S(r)+o(T(r,f)),\\
\label{2.19}
\bigl\|_E\ m\left( r,\dfrac{1}{f-1}\right)&\le \frac{1}{2\pi } \int\limits_{\theta_4(r)}\log^+ \dfrac{1}{|F|}d\theta+4S(r)+o(T(r,f)).
\end{align}
Combining (\ref{2.16})-(\ref{2.19}), we have
$$\bigl\|_E\ \sum_{i=2}^5m\left( r,\dfrac{1}{f-a_i}\right)\le m\left( r,\dfrac{1}{F}\right)+16S(r)+o(T(r,f)).$$
Therefore
\begin{align*}
\bigl\|_E\ 4T( r,f)\le\sum_{i=2}^5N(r,a_i,f)-N(r,0,F)+T(r,F)+16S(r)+o(T(r,f))
\end{align*}
Combining this inequality with (\ref{2.14}), we obtain
\begin{align}\label{2.20}
\bigl\|_E\ 4T( r,f)\le\sum_{i=2}^5\overline{N}(r,a_i,f)+T( r,F)+16S(r)+o(T(r,f)).
\end{align}
Also, from (\ref{2.5}), we have
$$m( r,F) \le 2m(r,f)+m\left(r,\frac{f'}{f}\right)+m\left(r,\frac{b_1'}{b_1}\right)+m\left(r,\frac{b_2'}{b_2}\right)+o(T(r,f))$$
and
$$ N( r,\infty,F)\le 2N(r,\infty,f)+\overline{N}(r,\infty,f)+ o(T(r,f)).$$
These yield that
\begin{align}\label{2.21}
\bigl\|_E\ T( r,F)\le 2T( r,f)+\overline{N}(r,\infty,f)+3S(r)+o(T(r,f)).
\end{align}
From (\ref{2.20}) and (\ref{2.21}), we have the conclusion of the theorem.
\end{proof} 

\section{Uniqueness theorem for meromorphic functions with finite growth index sharing five small functions}

In this section, we will proof Theorem \ref{1.2}. Firstly, we prove the following lemma.
\begin{lemma}\label{3.1}
Let $f$ and $g$ be two distinct meromorphic functions on $\Delta (R_0)$ with finite growth indices $c_f$ and $c_g$, respectively. Let $a_1,\ldots , a_q\ (q\ge 5)$ be distinct small functions with respect to $f$ and $g$. Suppose that 
$$\overline{E}(a_i,k,f)=\overline{E}(a_i,k,g)\ (1\le i\le q)$$ 
Let $\varepsilon$ be a positive real number. Setting $T(r)=T(r,f)+T(r,g)$, $\gamma(r)=e^{(\varepsilon +\max\{c_f,c_g\})T(r)}$ and $S(r)=(1+\varepsilon)\log\gamma (r)+\varepsilon\log r$, then we have
$$\bigl\|_E\ \sum_{i=5}^q\overline{N}(r,a_i,f,g)\le\sum_{i=1}^4\biggl (\overline{N}_{(k+1}(r,a_i,f)+\overline{N}_{(k+1}(r,a_i,g)\biggl )+7S(r)+o(T(r)).$$
\end{lemma}
Here, $\overline{N}(r,a,f,g)$ denotes the counting function without multiplicity which counts all common zeros of $(f-a)$ and $(g-a)$, and $\overline{N}_{(k+1}(r,a,f)$ denotes the counting function without multiplicity which counts zero of $(f-a)$ with multiplicity at least $k+1$.
\begin{proof}
We set $S=\bigcup_{1\le i<j\le q}\{z; a_i(z)=a_j(z)\}.$ Then we see that $S$ is a discrete subset of $\Delta (R_0)$ and $\overline{N}(r,S)=o(T(r))$.

By considering two functions $L(f)$ and $L(g)$ instead of $f$ and $g$ if necessary, where $L$ is the quasi-M\"{o}bius transformation $L(w)=\dfrac{(\omega -a_1)(a_3-a_2)}{(\omega -a_2)(a_3-a_1)}$, we may assume that $a_1=0,a_2=\infty, a_3=1$ and $ a_4=a$ with $a\not\in\{0,1,\infty\}$. Here, we note that the quasi-M\"{o}bius transformation $L$ only makes the counting functions in the inequality of the lemma change up to small terms $o(T(r))$. 

Consider the function
\begin{align}\label{3.2}
H= \dfrac{f'(a'g-ag')(f-g)}{(f(f-1)g(g-a))}-\dfrac{g'(a'f-af')(f-g)}{g(g-1)f(f-a)}. 
\end{align} 
We rewrite $H$ as
\begin{align}\label{3.3}
H=\dfrac{(f-g)Q}{f(f-1)(f-a)g(g-1)(g-a)},
\end{align}
where
\begin{align}\label{3.4}
\begin{split}
Q&=f'(a'g-ag')(f-a)(g-1)-g'(a'f-af')(g-a)(f-1)\\
&=a'ff'g^2-a'ff'g-a(a-1)ff'g'-aa'f'g^2+aa'f'g\\
&-a'f^2gg'+a'fgg'+a(a-1)f'gg'+aa'f^2g'-aa'fg'.
\end{split}
\end{align}

\textbf{Case 1:} Suppose that $H\equiv 0$. Then from (\ref{3.2}) we have
$$ \dfrac{f'(a'g-ag')}{(f-1)(g-a)}\equiv\dfrac{g'(a'f-af')}{(g-1)(f-a)}. $$
This implies that
\begin{align*} \dfrac{(f-g)(1-a)}{(g-1)(f-a)}&=\dfrac{(f-1)(g-a)}{(g-1)(f-a)}-1=\dfrac{f'(a'g-ag')}{g'(a'f-af')}-1\\
&=\dfrac{a'[(f'-g')g-(f-g)g']}{g'(a'f-af')}.
\end{align*}
It yields that
\begin{align}\label{3.5}
\dfrac{f'-g'}{f-g}=\dfrac{(1-a)g'(a'f-af')}{a'g(g-1)(f-a)}+\dfrac{g'}{g}.
\end{align}
Hence, if there exists a point $z_0\not\in S$ which is a common zero of $(f-a_i)$ and $g-a_i\ (5\le i\le q)$ then it must be a pole of the left hand side of (\ref{3.5}) but not be a pole of the right hand side. This is a contradiction. Therefore each common zero of $(f-a_i)$ and $(g-a_i) \ (5\le i\le q)$ must belong to $S$. Thus
$$\sum_{i=5}^q\overline{N}(r,a_i,f,g)\le (q-4)\overline{N}(r,S)=o(T(r)),$$
where $\overline{N}(r,S)$ is the counting function without multiplicity which counts all points in $S$. Hence, the lemma is proved in this case.

\textbf{Case 2:} Suppose that $H\not\equiv 0$. From (\ref{3.2}) and (\ref{3.3}), we easily see that if $z\not\in S$ is a common zero of $(f-a_i)$ and $(g-a_i) \ (5\le i\le q)$ then it is a zero of $(f-g)$ and is not a pole of $\frac{Q}{f(f-1)(f-a)g(g-1)(g-a)}$, and hence it is a zero of $H$. Therefore we have
\begin{align}\label{3.6}
\begin{split}
\sum_{i=5}^q\overline{N}(r,a_1,f,g)&\le \overline{N}(r,0,H)+(q-4)\overline{N}(r,S)\\
&\le T(r,H)+o(T(r))\\
&\le m(r,H)+N(r,\infty,H)+o(T(r)).
\end{split}
\end{align}

We now estimate the proximity function $m(r,H)$. Firstly, we have
\begin{align}\label{3.7}
\begin{split}
H=&\dfrac{f'}{f-1}\dfrac{a'g-ag'}{g(g-a)}-(\dfrac{f'}{f-1}-\dfrac{f'}{f})\dfrac{a'g-ag'}{g-a}\\
&-\dfrac{g'}{g-1}\dfrac{a'f-af'}{f(f-a)}-(\dfrac{g'}{g-1}-\dfrac{g'}{g})\dfrac{a'f-af'}{f-a}\\
=&\dfrac{f'}{f-1}(\dfrac{g'}{g}-\dfrac{g'-a'}{g-a})-(\dfrac{f'}{f-1}-\dfrac{f'}{f})(a'-a\dfrac{g'-a'}{g-a})\\
&-\dfrac{g'}{g-1}(\dfrac{f'}{f}-\dfrac{f'-a'}{f-a})-(\dfrac{g'}{g-1}-\dfrac{g'}{g})(a'-a\dfrac{f'-a'}{f-a}).
\end{split}
\end{align}
By the lemma on logarithmic derivatives, we get

\begin{align}\label{3.8}
\begin{split}
m(r,H)&\le m\left(r,\dfrac{f'}{f}\right)+m\left(r,\dfrac{g'}{g}\right)+m\left(r,\dfrac{f'}{f-1}\right)+m\left(r,\dfrac{g'}{g-1}\right)\\
&+m\left(r,\dfrac{(f-a)'}{f-a}\right)+m\left(r,\dfrac{(g-a)'}{g}\right)+m\left(r,\dfrac{a'}{a}\right)+o(T(r))\\
&\le 7S(r)+o(T(r)).
\end{split}
\end{align}

We now estimate the counting function $N(r,\infty,H)$. From (\ref{3.7}), we see that each pole of $H$ must be a zero of some functions $(f-a_i)$ or $(g-a_i)\ (i=1,\ldots ,4)$ (here we regard $f-\infty =\frac{1}{f}$) or a pole of  $a'$ or $a$. We consider the following four subcases.

\textit{Subcase 1:} $z$ is a pole of $a'$ or $a$. Hence $z$ must be a pole of $a$. We note that each pole of every meromorphic function of the form $\dfrac{h'}{h}$ has multiplicity at most 1. Therefore, from (\ref{3.7}) we have
$$ \nu^{\infty}_H(z)\le \nu^\infty_a(z)+2\le 3\nu^\infty_a(z). $$

\textit{Subcase 2:} $z$ is not a pole of $a$ and $z$ is a common zero of $(f-u)$ and $(g-u)$ for a function $u\in\{0,1,a\}$. From (\ref{3.3}), we rewrite $H$ as follows
\begin{align*}
H&=(f-g)\left [ \bigl (\dfrac{f'}{f-1}-\dfrac{f'}{f}\bigl )\bigl (\dfrac{g'}{g}-\dfrac{g'-a'}{g-a}\bigl )- \bigl (\dfrac{g'}{g-1}-\dfrac{g'}{g}\bigl )\bigl (\dfrac{f'}{f}-\dfrac{f'-a'}{f-a}\bigl )\right ]=(f-g)P,
\end{align*} 
where
$$P=\biggl [ \dfrac{f'}{f-1}\dfrac{g'}{g}- \dfrac{f'}{f-1}\dfrac{g'-a'}{g-a}+\dfrac{f'}{f}\dfrac{g'-a'}{g-a}-\dfrac{g'}{g-1}\dfrac{f'}{f}+\dfrac{g'}{g-1}\dfrac{f'-a'}{f-a}-\dfrac{f'-a'}{f-a}\dfrac{g'}{g}\biggl ].$$
Hence, $z$ is a zero of $(f-g)$ and is a simple pole of $P$. Therefore $z$ is not a pole of $H$.

\textit{Subcase 3:} $z$ is not a pole of $a$ and is a common pole of $f$ and $g$. From (\ref{3.3}) and (\ref{3.4}), we easily see that $z$ is not a pole of $H$.

\textit{Subcase 4:} $z$ is not a pole of $a$ and $z$ is either a zero of $f-a_i$ or a zero of $g-a_i$ for some $i\in\{1,\ldots ,4\}$. From (\ref{3.7}), we see that $H$ has the following form
$$ H=\sum_{\underset{u\ne v}{u,v\in\{0,1,a\}}}a_{uv}\dfrac{(f-u)'}{(f-u)}\dfrac{(g-v)'}{(g-v)}, $$
where $a_{uv}$ are constants or $\pm a'$ or $\pm a$. Hence
\begin{align*}
\nu^\infty_H(z)&\le\max_{\underset{u\ne v}{u,v\in\{0,1,a\}}}(\nu^\infty_{\frac{(f-u)'}{(f-u)}}(z)+\nu^\infty_{\frac{(g-v)'}{(g-v)}}(z))\\
&\le\sum_{i=1}^4(\min\{1,\nu^0_{f-a_i}(z)\}+\min\{1,\nu^0_{g-a_i}(z)\}-\min\{1,\nu^0_{f-a_i}(z)\cdot \nu^0_{g-a_i}(z)\}).
\end{align*}
From the above four case, we have
$$ \nu^\infty_H\le 3\nu^\infty_a +\sum_{i=1}^4(\min\{1,\nu^0_{f-a_i}\}+\min\{1,\nu^0_{g-a_i}\}-\min\{1,\nu^0_{f-a_i}\cdot \nu^0_{g-a_i}\}). $$
This yield that
\begin{align*}
N(r,\infty,H)&\le \sum_{i=1}^4(\overline{N}(r,a_i,f)+\overline{N}(r,a_i,f)-\overline{N}(r,a_i,f,g)) +o(T(r))\\
&\le\sum_{i=1}^4(\overline{N}_{(k+1}(r,a_i,f)+\overline{N}_{(k+1}(r,a_i,g))+o(T(r)).
\end{align*}
Combining the above inequality and (\ref{3.6}) and (\ref{3.8}) we get
$$\sum_{i=5}^q\overline{N}(r,a_1,f,g)\le\sum_{i=1}^4(\overline{N}_{(k+1}(r,a_i,f)+\overline{N}_{(k+1}(r,a_i,g))+7S(r)+o(T(r)).$$
Then the lemma is proved in this case.
\end{proof}

\begin{proof}[\sc Proof of Theorem \ref{1.2}] Suppose contrarily that $f\ne g$. We define $T(r),\gamma(r)$ and $S(r)$ as in Lemma \ref{3.1}. By Lemma \ref{3.1}, for every subset $\{i_1,i_3,i_4, i_4\}$ of $\{1,\ldots,q\}$ we have
\begin{align*}
\sum_{i=1}^q\overline{N}(r,a_i,f,g)&-\sum_{j=1}^4\overline{N}(r,a_{i_j},f,g)\\
&\le\sum_{j=1}^4\left (\sum_{i=1}^4\left(\overline{N}_{(k+1}(r,a_{i_j},f)+\overline{N}_{(k+1}(r,a_{i_j},g)\right)\right )+7S(r)+o(T(r)). 
\end{align*}
Summing-up both sides of the above inequality over all subsets $\{i_1,i_3,i_4, i_4\}$, we obtain
\begin{align*}
(q-4)\sum_{i=1}^q\overline{N}(r,a_i,f,g)\le4\sum_{i=1}^q\left (\overline{N}_{(k+1}(r,a_i,f)+\overline{N}_{(k+1}(r,a_i,g)\right )+7qS(r)+o(T(r)).
\end{align*}
Since $\overline{N}(r,a_i,f)+\overline{N}(r,a_i,g)\le \overline{N}(r,a_i,f,g)+\overline{N}_{(k+1}(r,a_i,f)+\overline{N}_{(k+1}(r,a_i,g)\ (1\le i\le q)$, the above inequality implies that
\begin{align*}(q-4)&\sum_{i=1}^q(\overline{N}(r,a_i,f)+\overline{N}(r,a_i,g))\\
&\le (q+4)\sum_{i=1}^q(\overline{N}_{(k+1}(r,a_i,f)+\overline{N}_{(k+1}(r,a_i,g))+7qS(r)+o(T(r))\\
&\le \frac{q+4}{k}\sum_{i=1}^q(k\overline{N}_{(k+1}(r,a_i,f)+\overline{N}_{(k+1}(r,a_i,g))+7qS(r)+o(T(r)).
\end{align*}
Since $k\overline{N}_{(k+1}(r,a_i,h)+\overline{N}(r,a_i,h)\le {N}(r,a_i,h)\ (i=1,\ldots,q; h=f,g)$, we easily have
\begin{align*}
\left(q-4+\frac{q+4}{k}\right)&\sum_{i=1}^q(\overline{N}(r,a_i,f)+\overline{N}(r,a_i,g))\\
&\le \frac{q+4}{k}\sum_{i=1}^q(N(r,a_i,f)+N(r,a_i,g))+7qS(r)+o(T(r))\\
&\le \frac{(q+4)q}{k}T(r)+7qS(r)+o(T(r)).
\end{align*}
By the second main theorem, we have
$$\bigl\|_E\ \left(q-4+\frac{q+4}{k}\right)\dfrac{q}{5}(2T(r)-38S(r))\le \frac{(q+4)q}{k}T(r)+7qS(r)+o(T(r)),$$
i.e.,
$$\bigl\|_E\ \dfrac{k(2q-8)-3(q+4)}{k(38q-117)+38(q+4)}T(r)\le S(r)+o(T(r)).$$
Letting $\varepsilon\rightarrow 0$, and then letting $r\rightarrow R_0\ (r\not\in E)$, then we obtain
$$ c_f+c_g\ge 2\min\{c_f,c_g\}\ge \dfrac{k(2q-8)-3(q+4)}{k(19q-\frac{117}{2})+19(q+4)}.$$
This contradiction completes the proof of the theorem.
\end{proof}

\section*{Disclosure statement}
This publication is supported by multiple datasets, which are available at locations cited
in the reference section.
%No potential conflict of interest was reported by the author(s).

\vskip0.2cm
{\footnotesize 
\noindent
{\sc Si Duc Quang}
\vskip0.05cm
\noindent
Department of Mathematics, Hanoi National University of Education,\\
136-Xuan Thuy, Cau Giay, Hanoi, Vietnam.
\vskip0.05cm
\noindent
\textit{E-mail}: quangsd@hnue.edu.vn

\end{document}